\documentclass[preprint]{elsarticle}

\usepackage{amsmath,amsfonts,amssymb,amsthm}
\usepackage{latexsym}

\numberwithin{equation}{section}

\newtheorem{deff}{Definition}[section]

\newtheorem{lemma}[deff]{Lemma}
\newtheorem{theorem}[deff]{Theorem}
\newtheorem{corollary}[deff]{Corollary}

\newtheorem{proposition}[deff]{Proposition}

\newtheorem{em-example}[deff]{Example}
\newtheorem{em-def}[deff]{Definition}        
\newtheorem{em-remark}[deff]{Remark}         
\newtheorem{em-question}[deff]{Question}

\newenvironment{example}{\begin{em-example} \em }{ \end{em-example}}
\newenvironment{remark}{\begin{em-remark} \em }{\end{em-remark}}

\newcommand{\N}{\mathbb N}

\newcommand{\CC}{\mathcal C}

\newcommand{\Z}{\mathbb Z}

\newcommand{\f}{\phi}
\def\ent{\mathrm{ent}}

\def\End{\mathrm{End}}
\def\Aut{\mathrm{Aut}}
\def\Im{\mathrm{Im}}

\def\K{\mathbb F_p}
\def\cont{\mathfrak c}
\global\def\hull#1{\langle{#1}\rangle}
\global\def\card#1{\left|{#1}\right|}

\global\def\dual#1{{#1}^*}

\input xy
\xyoption{all}

\journal{Journal of Algebra}

\begin{document}

\begin{frontmatter}

\title{Adjoint algebraic entropy}

\author{Dikran Dikranjan}
\ead{dikran.dikranjan@dimi.uniud.it}
\address{Dipartimento di Informatica e Matematica, Universit\`a di Udine, via delle Scienze, 206 - 33100 Udine, Italy}

\author{Anna Giordano Bruno}
\ead{anna.giordanobruno@math.unipd.it}
\address{Dipartimento di Matematica Pura ed Applicata, Universit\`a di Padova, Via Trieste, 63 - 35121 Padova, Italy}

\author{Luigi Salce}
\address{Dipartimento di Matematica Pura ed Applicata, Universit\`a di Padova, Via Trieste, 63 - 35121 Padova, Italy}
\ead{salce@math.unipd.it}

\begin{abstract}
The new notion of adjoint algebraic entropy of endomorphisms of Abelian groups is introduced. Various examples and basic properties are provided. It is proved that the adjoint algebraic entropy of an endomorphism equals the algebraic entropy of the adjoint endomorphism of the Pontryagin dual.  As applications, we compute the adjoint algebraic entropy of the shift endomorphisms of direct sums, and we prove an Addition Theorem for the adjoint algebraic entropy of bounded Abelian groups. A dichotomy is established,  stating that the adjoint algebraic entropy of any endomorphism can take only values zero or infinity.  As a consequence, we obtain the following surprising discontinuity criterion for endomorphisms: every endomorphism of a compact abelian group, having finite positive algebraic entropy, is  discontinuous. This resolves in a strong way an open question from \cite{AGB}. 
\end{abstract}


\end{frontmatter}

\section{Introduction}

The algebraic entropy of endomorphisms of Abelian groups was first defined by Adler, Konheim and McAndrew in 1965 \cite{AKM}, and then studied ten years later by Weiss in 1974/75 \cite{W}. Weiss proved that the algebraic entropy of an endomorphism $\phi : G \to G$ of a torsion Abelian group $G$ coincides with the topological entropy of the adjoint map $\dual\phi: \dual G\to \dual G$ of $\phi$, where $\dual G$ is the Pontryagin dual of $G$.
More recently, a deeper investigation of the algebraic entropy was developed by Dikranjan, Goldsmith, Salce and Zanardo in \cite{DGSZ}. Particular aspects of the algebraic entropy have been subsequently examined, as its relationship with the commutativity of the endomorphism ring modulo the ideal of the small endomorphisms (see \cite{SZ}), and its behaviour for generalized shifts on direct sums (see \cite{AADGH}) and direct products (see \cite{AGB}) of finite Abelian groups.

\smallskip
Roughly speaking, the algebraic entropy $\ent(\phi)$ of the group endomorphism $\phi : G \to G$ measures how much the finite subgroups $F$ of $G$ are moved by $\phi$. This measure is taken on the so-called $n$-th trajectories of $F$ (where $n$ is a natural number), namely, the finite subgroups of $G$ defined by $$T_n(\phi,F) = F + \phi F+ \phi^2F + \ldots+ \phi^{n-1}F$$ (while the $\phi$-trajectory of $F$ is $T(\phi,F)=\sum_{n}\phi^nF$,  the smallest $\phi$-invariant subgroup of $G$ containing $F$). More precisely, the exact definition of the algebraic entropy is:
$$\ent(\phi) = \sup\left\{ \lim_{n\to\infty} \frac{\log \card{T_n(\phi,F)}}{n}: F\leq G,\ \text{$F$ finite}\right\} .$$

It is natural to capture the action of the endomorphism $\phi$ making use of another class of subgroups of $G$, which is in a certain sense dual to the class of finite subgroups, namely, the class of subgroups of finite index, since this class is stable under taking inverse images of $\phi$. If $N$ is such a subgroup of $G$, it is natural to replace the $n$-th trajectory by the $n$-th cotrajectory of $N$, defined by
$$C_n(\phi,N) = \frac{G}{N\cap\phi^{-1}N\cap\ldots\cap\phi^{-n+1}N}$$
and this leads to the notion of adjoint algebraic entropy, defined by
$$\ent^\star(\phi)=\sup\left\{\lim_{n\to\infty}\frac{\log\card{C_n(\phi,N)}}{n}: N\leq G,\ G/N\ \text{finite}\right\}.$$

For many of the properties of the algebraic entropy the corresponding version for the adjoint algebraic entropy holds true, as we will see in Sections \ref{def-sec} and \ref{first}. However, this correspondence fails to behave as a perfect duality in certain points. For example, the value of the adjoint algebraic entropy of the Bernoulli shifts (see Example \ref{bernoulli-example} for their definitions) is always infinity, due to the following remarkable fact.  While a countable Abelian group has countably many finite subgroups, the size of the family of its finite-index subgroups may attain the cardinality of continuum or it is countable (the latter case is completely described in Theorem \ref{narrow:char}). 
The main consequence of this fact is that, while the algebraic entropy $\ent(\phi)$ can take values zero, positive (actually, logarithms of positive integers) or infinity, the adjoint algebraic entropy $\ent^\star(\phi)$ can take only values zero or infinite. This dichotomy is proved in Section \ref{dichotomy}.  This is one of the instances witnessing the failure
of a perfect ``duality" between these two entropies. 
Actually, a well-known duality plays a crucial role when dealing with the adjoint algebraic entropy, namely, the Pontryagin duality. In fact, we will show in Section \ref{sec-3} that the adjoint algebraic entropy $\ent^\star(\phi)$ equals the algebraic entropy of the adjoint map $\dual\phi: \dual G \to \dual G$, that is, the following formula holds:
$$\ent^\star(\phi) = \ent(\dual\phi).$$

In Section \ref{sec6}, as application of the above formula, we will compute the adjoint algebraic entropy of the Bernoulli shifts on direct sums and products, pointing out their different behaviour with respect to the algebraic entropy;
furthermore, we will prove the Addition Theorem for endomorphisms of bounded Abelian groups, explaining why its extension to wider classes of groups is impossible. Finally, in Section \ref{dichotomy}, after the proof of the dichotomy theorem, we will provide several characterizations of the endomorphisms with zero adjoint entropy.

All groups considered in this paper will be Abelian. Notation and unexplained terminology follow \cite{F}.

\section{Definition and existence}\label{def-sec}

We denote by $\N$ the set of natural numbers, by $\N_+$ the set of positive integers, and by $\Z$ the set of integers.
For $G$ an Abelian group, we denote by $\mathcal C(G)$ the set of the cofinite subgroups of $G$, that is, of the subgroups of finite index.

\begin{lemma}\label{trivial}
Let $H, G$ be Abelian groups and $f:H\to G$ a homomorphism. Then $f^{-1}N\in\CC(H)$ for every $N\in \CC(G)$. The correspondence $\Psi: \CC(G) \to \CC(H)$ defined by $N\mapsto f^{-1}(N)$ is injective if $f$ is surjective. 
\end{lemma}
\begin{proof}
Let $N\in\CC(G)$. The homomorphism $\tilde f:H/f^{-1}N\to G/N$ is well-defined and injective, hence $f^{-1}N\in\CC(H)$.
Assume that $f$ is surjective and that $f^{-1}N=f^{-1}M$ for some $N,M\in\CC(G)$. Then $N=f(f^{-1}N)=f(f^{-1}M))=M$, since $f$ is surjective.
\end{proof}

The correspondence $\Psi: N\mapsto f^{-1}N$ of Lemma \ref{trivial} need not be neither injective nor surjective, when $H$ is a subgroup of $G$ and $f$ is the inclusion $f: H\hookrightarrow G$. This need not be true even when the quotient $G/H$ is divisible. It is possible to see that, if $H$ is in addition pure on $G$, then $\Psi$ becomes a bijection. Moreover, if $H$ is a direct summand of $G$, then $\Psi$ is surjective (see the forthcoming paper \cite{SZ2}).

\medskip
Let $G$ be an Abelian group and $N\in \mathcal C(G)$.
For an endomorphism $\phi:G\to G$ and a positive integer $n$, let $$B_n(\phi,N)=N\cap\phi^{-1}N\cap\ldots\cap\phi^{-n+1}N$$ and let $$C_n(\phi,N)=\frac{G}{B_n(\phi,N)};$$ 
$C_n(\phi,N)$ is called the \emph{$n$-th $\phi$-cotrajectory} of $N$.
Let $$B(\phi,N)=\bigcap_{n\in\N}\phi^{-n}N,\ \ C(\phi,N)=\frac{G}{B(\phi,N)};$$ 
$C(\phi,N)$ is called the \emph{$\phi$-cotrajectory} of $N$.
It is easy to check that $B(\phi,N)$ is the maximum $\phi$-invariant subgroup of $N$.

Since the map induced by $\phi^n$ on $\frac{G}{\phi^{-n}N}\to \frac{G}{N}$ is injective, it follows that $\frac{G}{\phi^{-n}N}$ is finite for every $n\in\N$. Moreover, $C_n(\phi,N)$ is finite for every $n\in\N_+$, as $B_n(\phi,N)\in \mathcal C(G)$ for every $n\in\N_+$, because $\mathcal C(G)$ is closed under finite intersections. 

\smallskip
The \emph{adjoint algebraic entropy of $\phi$ with respect to $N$} is 
\begin{equation}\label{H*}
H^\star(\phi,N)={\lim_{n\to +\infty}\frac{\log|C_n(\phi,N)|}{n}}.
\end{equation}
We will show now that this limit exists and is finite.
Let $B_n=B_n(\phi,N)$, $C_n=C_n(\phi,N)$, $c_n=\card{C_n}$ and $\gamma_{n}=\log c_n$ for every $n\in\N_+$.
Since $B_{n+1}$ is a subgroup of $B_n$, it follows that 
$$C_{n}=\frac{G}{B_{n}}\cong\frac{\frac{G}{B_{n+1}}}{\frac{B_n}{B_{n+1}}}=\frac{C_{n+1}}{\frac{B_n}{B_{n+1}}}.$$ 
Then $c_{n} | c_{n+1}$ for every $n\in\N_+$. So, setting $\alpha_{n+1}= c_{n+1}/{c_n}$, we have: $$\alpha_{n+1}=\card{\frac{B_n}{B_{n+1}}}\in \N_+.$$

\begin{lemma}\label{alpha_n+1|alpha_n}
For each $n>1$, $\alpha_{n+1}$ divides $\alpha_n$ in $\N_+$.
\end{lemma}
\begin{proof}
Fix $n\in\N$, $n>1$.  We intend to prove that $\frac{B_n}{B_{n+1}}$ is isomorphic to a subgroup of $\frac{B_{n-1}}{B_n}$ and so $\alpha_{n+1}|\alpha_n$.

First note that $\frac{B_n}{B_{n+1}}\cong \frac{B_n+\phi^{-n}N}{\phi^{-n}N}$.
From $B_n=N\cap \phi^{-1}B_{n-1}\leq\phi^{-1}B_{n-1}$, it follows that $\frac{B_n+\phi^{-n}N}{\phi^{-n}N}\leq A_n=\frac{\phi^{-1}B_{n-1}+\phi^{-n}N}{\phi^{-n}N}$. Since the homomorphism $\tilde\phi:\frac{G}{\phi^{-n}N}\to \frac{G}{\phi^{-n+1}N}$, induced by $\phi$, is injective, also its restriction to $A_n$ is injective, and the image of $A_n$
is contained in $L_n= \frac{B_{n-1}+\phi^{-n+1}N}{\phi^{-n+1}N}$, which is isomorphic to $\frac{B_{n-1}}{B_n}$.
Summarizing,
$$\frac{B_n}{B_{n+1}}\cong\frac{B_n+\phi^{-n}N}{\phi^{-n}N}\leq A_n \rightarrowtail L_n\cong \frac{B_{n-1}}{B_n},$$
which concludes the proof.
\end{proof}

Clearly, $c_n=c_1\cdot\alpha_2\cdot\ldots\cdot\alpha_n$, as $c_{n+1}=\alpha_{n+1}\cdot c_n$, for every $n\in\N_+$. By Lemma \ref{alpha_n+1|alpha_n}, there exist $N\in\N_+$ and $\alpha\in\N_+$ such that $\alpha_n=\alpha$ for every $n\geq N$. Let $a_0=c_1\cdot \alpha_2\cdot \ldots\cdot \alpha_N$, so that $c_{n+1}=c_n\alpha$ 
(and consequently, $\gamma_{n+1}=\gamma_n + \log \alpha$ for all $n\geq N$.
Therefore,
\begin{equation}\label{(*)}
c_n=a_0\cdot\alpha^{n-N}, \mbox{ and so } \gamma_n= \log a_0 + (n-N)\log \alpha,  \mbox{ for all } n\geq N.
\end{equation}
From \eqref{(*)} one can immediately see that $\lim \frac{\gamma_n}{n} = \log \alpha$, i.e., 
 the limit in \eqref{H*} exists and coincides with $\log\alpha$. 

Note that in \eqref{(*)} either $\alpha> 1$ or $\alpha=1$; in the latter case,  the sequence $0<c_1\leq c_2\leq\ldots$ is stationary, or equivalently $B(\phi,N)=B_n(\phi,N)$ for some $n\in\N_+$. We resume all what is proved above in the following

\begin{proposition}\label{C=C_n->ent*=0} 
Let $G$ be an Abelian group, $\phi\in\End(G)$ and $N\in\CC(G)$. Then 
the sequence $\alpha_n=\card{\frac{C_{n+1}}{C_n}}=\card{\frac{B_n}{B_{n+1}}}$ is stationary; more precisely, there exists a natural number $\alpha>0$ such that $\alpha_n= \alpha$ for all $n$ large enough. 
In particular,
\begin{itemize}
\item[(a)] the limit in \eqref{H*} does exist and it coincides with $\log\alpha$, and
\item[(b)] $H^\star(\phi,N)=0$ if and only if $C(\phi,N)=C_n(\phi,N)$ for some $n\in\N_+$.\hfill$\qed$
\end{itemize}
\end{proposition}

By Proposition \ref{C=C_n->ent*=0}, $\ent^\star(\phi)=\infty$ if and only if there exists a countable family $\{N_k\}_{k\in\N}\subseteq\CC(G)$ such that $H^\star(\phi,N_k)$ converges to $\infty$.

\begin{remark} 
Another consequence of Lemma \ref{alpha_n+1|alpha_n} is that the sequence of non-negative reals $\{\gamma_n\}_{n\in\N_+}$ is subadditive, i.e., for every $n,m\in\N_+$:
\begin{equation}\label{(dag)}
\gamma_{n+m}\leq \gamma_n+\gamma_m.
\end{equation}
Indeed, it suffices to see that for every $n,m\in\N_+$, $c_{n+m}\leq c_{n}\cdot c_m$.
Write $c_{n+m}=c_1\cdot\alpha_2\cdot\ldots\cdot\alpha_n\cdot\alpha_{n+1}\cdot\ldots\cdot\alpha_{n+m}$, and $c_n\cdot c_m=c_1^2\cdot(\alpha_2\cdot\ldots\cdot\alpha_n)\cdot(\alpha_2\cdot\ldots\cdot\alpha_m)$. Comparing term by term, since $\alpha_2\geq\alpha_{n+1},\ldots,\alpha_{m}\geq\alpha_{n+m}$, and $c_1\geq 1$, we can conclude that $c_n\cdot c_m\geq c_{n+m}$, and consequently \eqref{(dag)} holds.
Using \eqref{(dag)} and a well-known result of Calculus, one can obtain an alternative proof (making no recourse to \eqref{(*)}), that  the limit in \eqref{H*} exists and coincides with $\inf_{n\in\N_+} \frac{\gamma_n}{n}$. The advantage of our approach (via \eqref{(*)}) is a precise information on the value of the limit, namely, it is a logarithm of a natural number (whereas this remains unclear using the argument based on Calculus). 
\end{remark}

An easy computation shows that $H^\star(\phi,M)$ is a anti-monotone function on $M$.

\begin{lemma}\label{N<M->HN>HM}
Let $G$ be an Abelian group, $\phi\in\End(G)$ and $N,M\in\CC(G)$. If $N\leq M$, then $B_n(\phi,N)\leq B_n(\phi,M)$ and so $|C_n(\phi,N)|\geq|C_n(\phi,M)|$. Therefore, $H^\star(\phi,N)\geq H^\star(\phi,M)$. \hfill$\qed$
\end{lemma}

Now we can define \emph{adjoint algebraic entropy} of $\phi:G\to G$ the quantity
$$\ent^\star(\phi)=\sup\{H^\star(\phi,N): N\in \CC(G)\},$$ and \emph{adjoint algebraic entropy} of $G$ the quantity $$\ent^\star(G)=\sup\{\ent^\star(\phi):\phi\in\End(G)\};$$ both $\ent^\star(\phi)$ and $\ent^\star(G)$ are either non-negative real numbers or $\infty$. More precisely, by Proposition \ref{C=C_n->ent*=0}, $\ent^\star(\phi)$ and $\ent^\star(G)$ are either logarithms of positive integers or $\infty$.

\section{Examples}\label{first}

In this section we provide various examples of computation of the adjoint algebraic entropy of endomorphisms of Abelian groups. It is natural to order the examples according to the size and complexity of the family $\CC(G)$ of all subgroups of finite index of $G$.
Recall that the \emph{profinite} (or \emph{finite index}) \emph{topology} $\gamma_G$ of the Abelian group $G$ is the topology having $\mathcal C(G)$ as a base of neighborhoods of $0$, while the \emph{natural} (or \emph{$\Z$-adic}) \emph{topology}  $\nu_G$ of $G$ is the topology having as a base of neighborhoods of $0$ the countable family of subgroups $\{mG\}_{m\in \N_+}$ (see \cite{F}). Since every finite-index subgroup of $G$ contains a subgroup of the form $mG$, one has always $\gamma_G \leq \nu_G$.

\smallskip
For an Abelian group $G$, let us consider $G^1=\bigcap_{n\in\N}n G$, the \emph{first Ulm subgroup}, which is fully invariant in $G$. It is well-known (see \cite{F}) that $$G^1=\bigcap\{N\leq G:G/N\ \text{is finite}\}.$$
  It is easy to see that $G^1=0$ precisely when $\gamma_G$ (and consequently, also $ \nu_G$) is Hausdorff. 

\medskip
Let us start with the trivial case when $\CC(G)$ is as small as possible. 
 
\begin{example}\label{example1}
A group $D$ is divisible if and only if  $\CC(D)=\{D\}$. In this case $\ent^\star(\phi)=0$ for every $\phi\in\End(D)$, that is $\ent^\star(D)=0$, since $D$ has no proper subgroups of finite index.
\end{example}

The next lemma is needed in the proof of Theorem \ref{narrow:char}.

\begin{lemma}\label{>c}
Let $p$ be a prime and $G$ an infinite $p$-bounded Abelian group. Then $|\CC(G)|\geq\mathfrak c$.
\end{lemma}
\begin{proof}
Since $G$ is infinite, $G\cong\Z(p)^{(\kappa)}$ for some cardinal $\kappa\geq\omega$.
Fixed any subset $X$ of $\kappa$, let $\phi_X:G\to \Z(p)$ be the map defined by $\phi_X((x_i)_{i<\kappa})=\sum_{i\in X}x_i$. Then $\ker\phi_X\in\CC(G)$ and $X\neq X'$, for $X,X'$ subsets of $\kappa$, implies $\ker\phi_X\neq \ker\phi_{X'}$. This shows that $G$ has at least $\mathfrak c$ different subgroups of finite index.
\end{proof}

It is an easy exercise to prove that $\CC(G)$ is finite if and only if $G/d(G)$ is a finite group, where $d(G)$ denotes the maximal divisible subgroup of $G$.

We characterize now the Abelian groups $G$ such that $\CC(G)$ is countable.

\begin{theorem}\label{narrow:char}
For an Abelian group $G$ the following conditions are equivalent: 
\begin{itemize}
\item[(a)] $\CC(G)$ is countable; 
\item[(b)] $|\CC(G)|< \cont$; 
\item[(c)] $G/pG$ is finite for every prime $p$;
\item[(d)] $G/mG$ is finite for every $m\in\N_+$;
\item[(e)] $\CC(G)$ contains a countable decreasing cofinal chain;
\item[(f)] the natural topology $\nu_G$ of $G$ coincides with the profinite topology $\gamma_G$ of $G$.
\end{itemize}
\end{theorem}
\begin{proof}
(a)$\Rightarrow$(b) and (d)$\Rightarrow$(c) are obvious.

\smallskip
(b)$\Rightarrow$(c) Assume that there exists a prime $p$ such that $G/pG$ is infinite. Then $|\CC(G/pG)|\geq \cont$ by Lemma \ref{>c}. Since the natural projection $G\to G/pG$ is surjective, Lemma \ref{trivial} yields $|\CC(G)|\geq \cont$. 

\smallskip
(c)$\Rightarrow$(d) Suppose that $G/pG$ is finite for every prime $p$. Consider $\dot p:G\to G$, the map induced by the multiplication by $p$. This endomorphism of $G$ induces $\overline{\dot p}:G/pG\to pG/p^2G$, which is well-defined and surjective. Consequently $pG/p^2G$ is finite; this implies that $G/p^2 G$ is finite. By induction it is possible to prove that $G/p^n G$ is finite for every $n\in\N_+$. Let $m\in\N_+$; then $m=p_1^{k_1}\cdot\ldots\cdot p_r^{k_r}$ for some primes $p_1,\ldots,p_r$ and some $k_1,\ldots, k_r\in\N_+$. Since $G/mG\cong G/p_1^{k_1}G\oplus\ldots\oplus G/p_r^{k_r}G$, it follows that $G/mG$ is finite.

\smallskip
(d)$\Rightarrow$(e) The required cofinal chain is $G\geq 2! G\geq 3!G\geq\ldots\geq n!G\geq\ldots$.

\smallskip
(e)$\Rightarrow$(a) Let $G\geq N_1\geq N_2\geq \ldots\geq N_n\geq\ldots$ be a countable decreasing cofinal chain in $\CC(G)$. For every $n\in\N_+$ let $\CC_n(G)=\{H\leq G: H\geq N_n\}$. Then $\CC_n(G)$ is finite for every $n\in\N_+$ and $\CC(G)=\bigcup_{n\in\N}\CC_n(G)$ is countable.

\smallskip
(d)$\Leftrightarrow$(f) If $G/mG$ is finite for every $m\in\N_+$, then $m G\in\CC(G)$ and so $\nu_G\leq\gamma_G$. Since in general $\gamma_G\leq\nu_G$, the two topologies coincide. Conversely, if $\nu_G\leq \gamma_G$, then for every $m\in\N_+$ there exists $N\in\CC(G)$ such that $mG\supseteq N$ and so $G/mG$ is finite.
\end{proof}

Call a group $G$ satisfying the above equivalent conditions {\em narrow}. It is easy to see that an Abelian group $G$ is narrow if and only if $G/d(G)$ is narrow.  
A reduced narrow Abelian group $G$ has the first Ulm subgroup $G^1=\bigcap_{n\in\N_+}nG$ vanishing; in fact, $G/pG$ finite for each prime $p$ implies that the $p$-primary component $t_p(G)$ is finite, hence $p^\omega G=\bigcap_{n\in\N_+}p^n G=0$, and consequently $G^1=0$.

\medskip
Let now $G$ be a reduced narrow Abelian group. Then its completion $\widetilde G$ with respect to the natural topology is compact and the completion topology coincides with its natural topology. The Abelian groups that are compact in their natural topology were described by Orsatti in \cite{O1} and investigated also in \cite{O2}. They have the form $G=\prod_p G_p$, where each $G_p$ is a finitely generated $\mathbb J_p$-module.

\medskip
We show below in Proposition \ref{narrow} that  $\ent^\star(G)=0$ for every narrow Abelian group. To this end we need the following lemma supplying a large scale of examples of endomorphisms with zero adjoint entropy. 

\begin{lemma}\label{invariant->H=0}
Let $G$ be an Abelian group, $\phi\in\End(G)$ and $A$ a $\phi$-invariant subgroup of $G$ of finite index. Then $H^\star(\phi,A)=0$.
\end{lemma}
\begin{proof}
Since $\phi A\subseteq A$, it follows that $\phi^{-1} A\supseteq A$. Then $B_n(\phi,A)=A$ for every $n\in\N_+$ and so $C_n(\phi,A)=G/A$ for every $n\in\N_+$. Consequently $H^\star(\phi,A)=0$.
\end{proof}

Now we apply the above lemma to get some easy examples.

\begin{example}\label{multiplication}
(a) For $G$ any Abelian group, $\ent^\star(id_G)=\ent^\star(0_G)=0$, by a trivial application of Lemma \ref{invariant->H=0}. Moreover, if an endomorphism $\phi:G\to G$ is nilpotent (i.e., there exists $n\in\N_+$ such that $\phi^n=0$) or periodic (i.e., there exists $n\in\N_+$ such that $\phi^n=id_G$), then $\ent^\star(\phi)=0$; this can be easily proved directly, or is a consequence of the next Lemma \ref{ll}.

\smallskip
(b) Let $G$ be an Abelian group, $m\in\Z$ and $\dot{m}:G\to G$ the endomorphism of $G$ defined by $x\mapsto mx$ for every $x\in G$. Then Lemma \ref{invariant->H=0} shows that $\ent^\star(\dot{m})=0$, since all subgroups of $G$ are all $\dot{m}$-invariant.
\end{example}

\begin{example}
Let $G$ be an Abelian group such that every subgroup of finite index contains a fully invariant subgroup of finite index. Then $\ent^\star(G)=0$.
Indeed, let $\phi\in\End(G)$ and let $N\in\CC(G)$. By hypothesis $N$ contains a fully invariant subgroup $A$ of $G$. In particular, $A$ is $\phi$-invariant and so $H^\star(\phi,A)=0$ by Lemma \ref{invariant->H=0}. By Lemma \ref{N<M->HN>HM} $H^\star(\phi,N)=0$ as well. Since $N$ is arbitrary, this proves that $\ent^\star(\phi)=0$.
\end{example}

\begin{proposition}\label{narrow}
Every narrow group $G$ satisfies $\ent^\star(G)=0$. 
\end{proposition}
\begin{proof}
If $N\in \CC(G)$, then there exists $m>0$ such that $N\geq mG$. Since $mG\in \CC(G)$ (as $G$ is narrow), and since $mG$ is fully invariant, we can apply Lemma \ref{invariant->H=0} to conclude that $\ent^\star(G)=0$. 
\end{proof}

As an application of the previous example we verify that the group of $p$-adic integers $\mathbb J_p$, for $p$ a prime, and every finite-rank torsion-free Abelian group have zero adjoint algebraic entropy.

\begin{example}\label{tf,fr->ent0}
(a) For $p$ a prime, $\ent^\star(\mathbb J_p)=0.$ Indeed, every subgroup of finite index of $\mathbb J_p$ contains some $p^n\mathbb J_p$, having finite index. Thus $\mathbb J_p$ is narrow, so Proposition \ref{narrow} applies.

\smallskip
(b) If $G$ is a torsion-free Abelian group of finite rank, then $\ent^\star(G)=0$.
In fact, for $N\in\CC(G)$, there exists $m\in\N_+$ such that $mG\leq N$. It is well-know that for an Abelian group $G$ of finite rank, $G/mG$ is finite. Then Proposition \ref{narrow} again applies.
\end{example}

Here is another example dealing with torsion-free Abelian groups.

\begin{example}\label{End(G)<Q}
Let $G$ be a torsion-free Abelian group. 

\smallskip
Assume first that $G=nG$ for some non-zero $n\in\Z$. Let us consider the endomorphism $(\frac{m}{n})^\cdot:G\to G$ defined by $x\mapsto \frac{m}{n}x$ for every $x\in G$, where $m\in\Z$ and $(m,n)=1$. 
We show that $\ent^\star((\frac{m}{n})^\cdot)=0$.
Let $N\in\CC(G)$. Then there exists $r\in\N_+$ such that $r G\subseteq N$. Since $n G=G$ we can suppose without loss of generality that $(r,n)=1$. Consequently, there exist $\gamma,\delta\in\N_+$ such that $1=\gamma r+\delta n$. Let $x\in N\cap n G$; then $x=n y$ for some $y\in G$. Since $y=\gamma r y+\delta n y=\gamma ry+\delta x\in N$, we get that $N\cap n G=n N$. But $n G =G$ and so $N=n N$. 
This shows that $N$ is $(\frac{m}{n})^\cdot$-invariant and Lemma \ref{invariant->H=0} implies that $H^\star((\frac{m}{n})^\cdot,N)=0$; hence $\ent^\star((\frac{m}{n})^\cdot)=0$.

\smallskip
From what we have seen above it follows that, if $\End(G)\subseteq \mathbb Q$, then $\ent^\star(G)=0$. 
\end{example}

Let $X$ be a set and $f:X\to X$ a function. A point $x\in X$ is \emph{quasi-periodic} for $f$ if there exist $s<t$ in $\N$ such that $f^s(x)=f^t(x)$. The function $f$ is \emph{locally quasi-periodic} if every point of $X$ is quasi-periodic and it is \emph{quasi-periodic} if $f^s=f^t$ for some $s<t$ in $\N$.

\begin{example}\label{qp->ent*=0}
Let $G$ be an Abelian group and $\phi\in\End(G)$. If $\phi$ is quasi-periodic, then $\ent^\star(\phi)=0$.
Indeed, by hypothesis there exist $n\in\N$, $m\in\N_+$ such that $\phi^n=\phi^{n+m}$. If $N\in\CC(G)$, then $\phi^{-n}N=\phi^{-n+m}N$ and so $B(\phi,N)=B_{n+m}(\phi,N)$. By Proposition \ref{C=C_n->ent*=0}(b), $H^\star(\phi,N)=0$, and since $N$ was arbitrary we conclude that $\ent^\star(\phi)=0$.
\end{example}

The converse implication of Example \ref{qp->ent*=0} does not hold true in general, as shown for example by any multiplication $\dot m:\Z\to \Z$, where $m\in\N$; in fact, $\ent^\star(\dot m)=0$ by Example \ref{multiplication}(b), but $\dot m$ is not quasi-periodic. But the converse implication holds true in the particular case of $p$-bounded Abelian groups, as shown by Theorem \ref{dich}, which, together with its corollary, will give the characterization of endomorphism of zero adjoint algebraic entropy.

\begin{example}\label{bernoulli-example}
Let $K$ be a non-zero Abelian group and consider:

\smallskip
(a) the \emph{right  Bernoulli shift} $\beta_K$ and the \emph{left Bernoulli shift} ${}_K\beta$ of the group $K^{\N}$ defined respectively by 
$$\beta_K(x_1,x_2,x_3,\ldots)=(0,x_1,x_2,\ldots)\ \mbox{and}\ {}_K\beta(x_0,x_1,x_2,\ldots)=(x_1,x_2,x_3,\ldots);$$

\smallskip
(a) the \emph{two-sided Bernoulli shift} $\overline{\beta}_K$ of the group $K^{\Z}$, defined by 
$$\overline\beta_K((x_n)_{n\in\Z})=(x_{n-1})_{n\in\Z}, \mbox{ for } (x_n)_{n\in\Z}\in K^{\Z}.$$

\smallskip
Since $K^{(\N)}$ is both $\beta_K$-invariant and ${}_K\beta$-invariant, and $K^{(\Z)}$ is $\overline\beta_K$-invariant, let $\beta_K^\oplus=\beta_K\restriction_{K^{(\N)}}$, ${}_K\beta^\oplus= {}_K\beta\restriction_{K^{(\N)}}$ and $\overline\beta_K^\oplus=\overline\beta_K\restriction_{K^{(\Z)}}$.

One can prove directly that $\ent^\star(\beta_K^\oplus)=\ent^\star({}_K\beta^\oplus)=\ent^\star(\overline\beta_K^\oplus)=\infty$; we prefer to postpone the proof, applying other results which simplify it (see Proposition \ref{bernoulli}).
\end{example}

Until now we have seen examples of group endomorphisms with adjoint algebraic entropy either zero or infinity. This behaviour is different to that of the algebraic entropy, which takes also finite positive values; for example, the right Bernoulli shift $\beta_K^\oplus:K^{(\N)}\to K^{(\N)}$ for a finite Abelian group $K$ has $\ent(\beta_K^\oplus)=\log|K|$ (see \cite[Example 1.9(a)]{DGSZ}). The reason of this fact will be clarified in Section \ref{dichotomy}.

\section{Basic properties}\label{basic}

In this section we prove basic properties of the adjoint algebraic entropy, in analogy with the properties of the algebraic entropy, that we collect in the following list (see \cite{W} and \cite{DGSZ}).
Let $G$ be an Abelian group and $\phi\in\End(G)$.
\begin{itemize}
\item[(A)] (Invariance under conjugation) If $\phi$ is conjugated to an endomorphism $\psi:H\to H$ of another Abelian group $H$, by an isomorphism, then $\ent(\phi) = \ent(\psi)$.
\item[(B)] (Logarithmic law) For every non-negative integer $k$, $\ent(\phi^k) = k \cdot \ent(\phi)$. If $\phi$ is an automorphism, then $\ent(\phi^k)=|k|\cdot\ent(\phi)$ for every integer $k$.
\item[(C)] (Addition Theorem) If $G$ is torsion group and $H$ is a $\phi$-invariant subgroup of $G$, then $\ent(\phi)=\ent(\phi\restriction_H)+\ent(\overline\phi)$, where $\overline\phi:G/H\to G/H$  is the endomorphism induced by $\phi$.
\item[(D)] (Continuity with respect to direct limits) If $G$ is direct limit of $\phi$-invariant subgroups $\{G_i : i \in I\}$, then $\ent(\phi)=\sup_{i\in I}\ent(\phi\restriction_{G_i})$.
\item[(E)] (Uniqueness) The algebraic entropy of endomorphisms of torsion Abelian groups is the unique collection $h = \{h_G : G \text{ torsion Abelian group}\}$ of functions $h_G:\End(G) \to \mathbb R_+$ satisfying (A), (B), (C), (D) and $h_{\Z(p)^{(\N)}}(\beta_{\Z(p)})=\log p$ for every  prime $p$ ($\mathbb R_+$ denotes the non-negative real numbers with the symbol $\infty$).
\end{itemize}

The Addition Theorem (C) for the algebraic entropy is one of the main results for the algebraic entropy, proved in \cite[Theorem 3.1]{DGSZ}.
As stated in (E), the set of four  properties (A), (B), (C), (D) is fundamental because it gives uniqueness for the algebraic entropy in the class of torsion Abelian groups. This theorem was inspired by the uniqueness theorem for the topological entropy of the continuous endomorphism of compact groups established by Stojanov \cite{St}.
 This motives us to study appropriate counterparts of these properties for the adjoint algebraic entropy.

\smallskip
Recall that for an Abelian group $G$ and $\phi\in\End(G)$, $\ent(\phi)=\ent(\phi\restriction_{t(G)})$. Moreover, it is proved in \cite[Theorem 1.1]{W} that for a torsion Abelian group $G=\bigoplus_p t_p(G)$, where $t_p(G)$ is the $p$-component of $G$, $\ent(\phi)=\sum_p \ent(\phi\restriction_{t_p(G)})$. Finally, \cite[Proposition 1.18]{DGSZ} shows that $\ent(\phi\restriction_{t_p(G)})=0$ if and only if $\ent(\phi\restriction_{G[p]})=0$, where $G[p]=\{x\in G: p x=0\}$ is the socle of $G$. From these results we immediately derive the next

\begin{lemma}\label{1.18}
Let $G$ be a torsion Abelian group and $\phi\in\End(G)$. Then $\ent(\phi)=0$ if and only if $\ent(\phi\restriction_{G[p]})=0$ for every prime $p$.\hfill $\qed$
\end{lemma}

For the sake of completeness we recall in the following proposition properties characterizing group endomorphisms of zero algebraic entropy.

\begin{proposition}\label{non-torsion->ent*=infty}\emph{\cite[Proposition 2.4]{DGSZ}}
Let $G$ be a $p$-bounded Abelian group. Then the following conditions are equivalent:
\begin{itemize}
\item[(a)]$\ent(\phi)=0$;
\item[(b)]$\phi$ is locally algebraic;
\item[(c)]$\phi$ is locally quasi-periodic;
\item[(d)]the $\phi$-trajectory of every $x\in G$ is finite.\hfill$\qed$
\end{itemize}
\end{proposition}

The above proposition shows the relevance of the notion of local   quasi periodicity in the context of endomorphisms of Abelian groups.

\medskip
The first property of the adjoint algebraic entropy that we prove is the analogue of property (A) above, that is, invariance under conjugation.

\begin{lemma}\label{cbi}
Let $G$ be an Abelian group and $\phi\in\End(G)$. If $H$ is another Abelian group and $\xi:G\to H$ an isomorphism, then $\ent^\star(\xi\circ\phi\circ\xi^{-1})=\ent^\star(\phi)$.
\end{lemma}
\begin{proof}
Let $N$ be a subgroup of $H$ and call $\theta=\xi\circ\phi\circ\xi^{-1}$. Then:
\begin{itemize}
\item[-] $H/N$ is finite if and only if $G/\xi^{-1}N$ is finite, and
\item[-] $B_n(\theta,N)=\xi(B_n(\phi,\xi^{-1}N)$ for every $n\in\N$, since $\theta^n=\xi\circ\phi^n\circ\xi^{-1}$ for every $n\in\N$.
\end{itemize}
Therefore, for every $n\in\N_+$, 
$$\card{C_n(\theta,N)}=\card{\frac{H}{B_n(\theta,N)}}=\card{\frac{\xi G}{\xi B_n(\phi,\xi^{-1}N)}}=\card{\frac{G}{B_n(\phi,\xi^{-1}N)}}=\card{C_n(\phi,\xi^{-1}N)}.$$ Consequently $H^\star(\theta,N)=H^\star(\phi,\xi^{-1}N)$ for every $N\in\CC(H)$, and hence $\ent^\star(\theta)=\ent^\star(\phi)$.
\end{proof}

The second property is a logarithmic law for the adjoint algebraic entropy, which is the counterpart of property (B) for the algebraic entropy.

\begin{lemma}\label{ll}
Let $G$ be an Abelian group and $\phi\in\End(G)$. Then $\ent^\star(\phi^k)=k\cdot \ent^\star(\phi)$ for every $k\in\N_+$.
\end{lemma}
\begin{proof}
For $N\in\CC(G)$, fixed $k\in\N$, for every $n\in\N$ we have
\begin{equation}\label{Cnkf=Cnfk}
C_{nk}(\phi,N )=C_n(\phi^k,B_k(\phi,N)).
\end{equation}
Then
\begin{align*}
k\cdot H^\star(\phi,N)&=k\cdot\lim_{n\to \infty}\frac{\log\card{C_{nk}(\phi,N)}}{nk}=\lim_{n\to \infty}\frac{\log\card{C_{nk}(\phi,N)}}{n}=\\
&=\lim_{n\to\infty}\frac{\log\card{C_n(\phi^k,B_k(\phi,N))}}{n}=H^\star(\phi^k,B_k(\phi,N))\leq\ent^\star(\phi^k).
\end{align*}
Consequently $k\cdot \ent^\star(\phi)\leq\ent^\star(\phi^k)$.

Now we prove the converse inequality. Indeed, for $N\in\CC(G)$ and for $k\in\N_+$,
\begin{align*}
\ent^\star(\phi)\geq H^\star(\phi,N)&=\lim_{n\to\infty}\frac{\log\card{C_{nk}(\phi,N)}}{nk}=\lim_{n\to\infty}\frac{\log\card{C_n(\phi^k,B_k(\phi,N))}}{nk}=\\
&=\frac{H^\star(\phi^k,B_k(\phi,N))}{k}\geq \frac{H^\star(\phi^k,N)}{k},
\end{align*}
where in the second equality we have applied \eqref{Cnkf=Cnfk} and in the last inequality Lemma \ref{N<M->HN>HM}. This proves that $k\cdot \ent^\star(\phi)\geq \ent^\star(\phi^k)$ and this concludes the proof.
\end{proof}

The next lemma shows that a group automorphism has the same adjoint algebraic entropy as its inverse.

\begin{lemma}
Let $G$ be an Abelian group and let $\phi\in\Aut(G)$. Then $\ent^\star(\phi)=\ent^\star(\phi^{-1})$.
\end{lemma}

\begin{proof}
For every $n\in\N_+$ and every $N\in\CC(G)$, we have
\begin{align*} 
\phi^{n-1}B_n(\phi,N) &=\phi^{n-1}(N\cap\phi^{-1}N\cap\ldots\cap\phi^{-n+1}N)=\\
&=\phi^{n-1}N\cap\phi^{n-2}N\cap\ldots\cap \phi N\cap N=B_n(\phi^{-1},N),
\end{align*}
and so 
$$\card{C_n(\phi,N)}=\card{\frac{G}{B_n(\phi,N)}}=\card{\frac{\phi^{n-1}G}{\phi^{n-1}B_n(\phi,N)}}=\card{\frac{G}{B_n(\phi^{-1},N)}}=\card{C_n(\phi^{-1},N)}.$$ 
Therefore $H^\star(\phi,N)=H^\star(\phi^{-1},N)$, and hence $\ent^\star(\phi)=\ent^\star(\phi^{-1})$.
\end{proof}

The following corollary is a direct consequence of the previous two results.

\begin{corollary}
Let $G$ be an Abelian group and $\phi\in\Aut(G)$. Then $\ent^\star(\phi^k)=|k|\cdot \ent^\star(\phi)$ for every $k\in\Z$. \hfill$\qed$
\end{corollary}

Now we give various weaker forms of the Addition Theorem for the adjoint algebraic entropy, inspired by the Addition Theorem for the algebraic entropy (see property (C) above). This theorem will be proved in the next section in Proposition \ref{AT*}.

The next property, a sort of monotonicity law for quotients over invariant subgroups, will be often used in the sequel.

\begin{lemma}\label{quotient}
Let $G$ be an Abelian group, $\phi\in\End(G)$ and $H$ a $\phi$-invariant subgroup of $G$. Then $\ent^\star(\phi)\geq \ent^\star(\overline\phi)$, where $\overline \phi:G/H\to G/H$ is the endomorphism induced by $\phi$.
\end{lemma}
\begin{proof}
Let $N/H\in\CC(G/H)$; then $N\in\CC(G)$. Fixed $n\in\N_+$,
$$B_n\left(\overline \phi, \frac{N}{H}\right)\geq \frac{B_n(\phi,N)+H}{H},$$
and so 
$$\left|C_n\left(\overline\phi,\frac{N}{H}\right)\right|\leq 
\left|\frac{\frac{G}{H}}{\frac{B_n(\phi,N)+H}{H}}\right|\leq
\left|\frac{G}{B_n(\phi,N)}\right|=|C_n(\phi,N)|.$$
This yields $H^\star(\overline\phi, N/H)\leq H^\star(\phi,N)$ and so $\ent^\star(\overline\phi)\leq\ent^\star(\phi)$.
\end{proof}

In general $\ent^\star(-)$ fails to be monotone with respect to taking restrictions to $\phi$-invariant subgroups (i.e.,  if $G$ is an Abelian group, $\phi\in\End(G)$ and $H$ a $\phi$-invariant subgroup of $G$, then the inequality $\ent^\star(\phi)\geq\ent^\star(\phi\restriction_H)$ may fail), as the following easy example shows. 

\begin{example}\label{noAT*}
Let $G$ be an Abelian group that admits a $\phi\in\End(G)$ with $\ent^\star(\phi)>0$. Then the divisible hull $D$ of $G$ has $\ent^\star(D)=0$ being divisible, by Example \ref{example1}.
\end{example}

This is not surprising, because this is the dual situation with respect to monotonicity for quotients of the algebraic entropy, which holds only for torsion Abelian groups (see \cite[fact (f) in Section 1]{DGSZ}).

Actually, the preceding properties show that the adjoint algebraic entropy has a behaviour which can be considered as dual   for many aspects to that of the algebraic entropy.
Furthermore, if we impose a strong condition on the invariant subgroup $H$, we obtain the searched monotonicity for invariant subgroups:

\begin{lemma}\label{subgroup}
Let $G$ be an Abelian group, $\phi\in\End(G)$ and $H$ a $\phi$-invariant subgroup of $G$. If $H\in\CC(G)$, then $\ent^\star(\phi)=\ent^\star(\phi\restriction_H)$.
\end{lemma}
\begin{proof}
Let $N\in\CC(H)$. Since $H$ has finite index in $G$, and since $G/H\cong (G/N)/(H/N)$, $N$ has finite index in $G$ as well. 
It is possible to prove by induction on $n\in\N_+$ that
$$B_n(\phi\restriction_H,N)=B_n(\phi,N)\cap H.$$
Then, for every $n\in\N_+$
\begin{align*}
C_n(\phi,N)&=\frac{G}{B_n(\phi,N)}\geq\frac{H+B_n(\phi,N)}{B_n(\phi,N)}\cong\\
&\cong\frac{H}{B_n(\phi,N)\cap H}=\frac{H}{B_n(\phi\restriction_H,N)}=C_n(\phi\restriction_H,N),
\end{align*}
and so
$\ent^\star(\phi)\geq H^\star(\phi,N)\geq H^\star(\phi\restriction_H,N)$, that implies $\ent^\star(\phi)\geq\ent^\star(\phi\restriction_H)$.
For $c_n=\card{C_n(\phi,N)}$ and $c_n'= \card{C_n(\phi\restriction_H,N)}$ we proved that $c_n \geq c_n'$. On the other hand, one can easily see that $c_n/c_n'\leq |G/H|$ is bounded. Therefore, $H^*(\phi\restriction_{H},N)=\lim_{n\to \infty}\frac{\log c_n'}{n}=\lim_{n\to\infty} \frac{\log c_n}{n}=H^\star(\phi,N)$. we can conclude that $\ent^\star(\phi)=\ent^\star(\phi\restriction_H)$.
\end{proof}

The next three results deal with the adjoint algebraic entropy of group endomorphism of direct sums.

\begin{lemma}\label{poorAT}
Let $G$ be an Abelian group.
If $G = G_1 \oplus G_2$ for some $G_1,G_2\leq G$, and $\f= \f_1 \oplus \f_2: G \to G$ for some $\f_1\in \End(G_1), \f_2\in \End(G_2)$, then $\ent^\star(\f) = \ent^\star(\f_1) + \ent^\star(\f_2)$.   
\end{lemma}
\begin{proof}
Use the fact that the subgroups $N=N_1\oplus N_2$ of $G$, where $N_i\in {\mathcal C}(G_i)$ for $i=1,2$, form a cofinal set in $C(G)$ and $C_n(\f, N)\cong C_n(\f_1,N_1)\oplus C_n(\f_2,N_2)$. 
\end{proof}

\begin{corollary}\label{DirSum1}
If $G = G_1 \oplus G_2$, then $\ent^\star(G)=0$ yields $\ent^\star(G_1)=\ent^\star(G_2)=0$. Moreover, if $\mathrm{Hom}(G_1,G_2)=0$ and $\mathrm{Hom}(G_2,G_1)=0$, then the inverse implication holds too.
\end{corollary}
\begin{proof}
The first claim follows from Lemma \ref{subgroup}, the latter from Lemma \ref{poorAT}.
\end{proof}

One can easily extend this corollary to direct sums of arbitrary families of Abelian groups. We give only a particular case that will be needed in the sequel. 

\begin{corollary}\label{DirSum2}
If $G =\bigoplus_p G_p$ is the primary decomposition of a torsion Abelian group $G$ and $\f\in \End(G)$, then 
$\ent^\star(\f) = \sum_p \ent^\star(\f \restriction _{G_p})$. In particular, $\ent^\star(G)=0$ if and only if $\ent^\star(\f \restriction _{G_p} )=0$ for all primes $p$. \hfill$\qed$
\end{corollary}

For an Abelian group $G$ and $\phi\in\End(G)$ define 
\begin{equation}\label{phi^1}
\phi^1:G/G^1\to G/G^1
\end{equation}
the endomorphism induced by $\phi$ on the quotient of $G$ on its first Ulm subgroup $G^1$.
Since $(G/G^1)^1=0$ for an Abelian group $G$, the following proposition shows that, in the computation of the adjoint algebraic entropy, we can always assume $G^1=0$.

\begin{proposition}\label{ent*=ent*1}
If $G$ is an Abelian group and $\phi\in\End(G)$, then $\ent^\star(\phi)=\ent^\star(\phi^1).$
\end{proposition}
\begin{proof}
By Lemma \ref{quotient} $\ent^\star(\phi)\geq\ent^\star(\phi^1)$. So let $N\in\CC(G)$. Since $N\supseteq n G$ for some $n\in\N_+$, $N\supseteq G^1$ as well, and $N/G^1\in\CC(G/G^1)$. We show that 
\begin{equation}\label{H*=H*}
H^\star(\phi,N)=H^\star(\phi^1,N/G_1)\leq\ent^\star(\phi^1).
\end{equation}
Since $N$ has finite index in $G$, we have seen that also $\phi^{-n}N$ has finite index in $G$ for every $n\in\N_+$, and so $\phi^{-n}N\supseteq G^1$ for every $n\in\N_+$. Consequently $B_n(\phi,N)\supseteq G^1$ for every $n\in\N_+$. Moreover, for every $n\in\N_+$,
$$\frac{B_n(\phi,N)}{G^1}=B_n\left(\phi^1,\frac{N}{G^1}\right);$$
in fact, 
\begin{align*}
\frac{B_n(\phi,N)}{G^1}&=\frac{N}{G^1}\cap \frac{\phi^{-1}N}{G^1}\cap \ldots\cap \frac{\phi^{-n+1}N}{G^1}=\\
&=\frac{N}{G^1}\cap(\phi^1)^{-1}\frac{N}{G^1}\cap\ldots\cap(\phi^1)^{-n+1}\frac{N}{G^1}=B_n\left(\phi^1,\frac{N}{G^1}\right).
\end{align*}
Therefore, for every $n\in\N_+$, $$C_n(\phi,N)=\frac{G}{B_n(\phi,N)}\cong\frac{\frac{G}{G^1}}{\frac{B_n(\phi,N)}{G^1}}=\frac{\frac{G}{G^1}}{B_n\left(\phi^1,\frac{N}{G^1}\right)}=C_n\left(\phi^1,\frac{N}{G^1}\right).$$
This implies that \eqref{H*=H*} holds true.
\end{proof}

Note that the property proved in Proposition \ref{ent*=ent*1} is dual to the property satisfied by the algebraic entropy: $\ent(\phi)=\ent(\phi\restriction_{t(G)})$, where $t(G)=\sum\{F\leq G:F\ \text{finite}\}$ is the torsion subgroup of $G$ (see \cite[fact (d) in Section 1]{DGSZ}).

\medskip
In analogy with the continuity of the algebraic entropy for direct limits of invariant subgroups (see property (D) above), one would expect to find continuity of the adjoint algebraic entropy for {\em inverse}   limits. Example \ref{no-lim} will show that this continuity fails, even in the case of endomorphisms of bounded Abelian groups.

\medskip
We consider now the behaviour of two commuting endomorphisms. This is inspired by the analogous result for the algebraic entropy in \cite[Lemma 2.5]{DGSZ}.

\begin{example}
Let $G$ be an Abelian group and let $\phi,\psi\in\End(G)$ be such that $\phi\psi=\psi\phi$. 

\smallskip
(a) an easy computation shows that, for every $n\in\N_+$ and every $N\in\CC(G)$, we have $$B_n(\phi,B_n(\psi,N))\leq B_n(\phi\psi,N)\cap B_n(\phi+\psi,N).$$

\smallskip
(b) If $\ent^\star(\psi)=0$, then $$\ent^\star(\phi\psi)\leq\ent^\star(\phi)\ \text{and}\ \ent^\star(\phi+\psi)\leq\ent^\star(\phi).$$ In fact, given $N\in\CC(G)$, $H^\star(\psi,N)=0$ by hypothesis and so $B(\psi,N)=B_m(\psi,N)$ for some $m\in\N_+$ and in particular $B(\psi,N)\in\CC(G)$. 
By (a) $B_n(\phi\psi,N)\geq B_n(\phi,B(\psi,N))$ for every $n\in\N_+$.
Consequently, $\card{C_n(\phi\psi,N)}\leq\card{C_n(\phi,B(\psi,N))}$ for every $n\in\N_+$, and this yields $H^\star(\phi\psi,N)\leq\ent^\star(\phi)$; hence $\ent^\star(\phi\psi)\leq\ent^\star(\phi)$. By (a) $\card{C_n(\phi+\psi,N)}\leq\card{C_n(\phi,B(\psi,N))}$ for every $n\in\N_+$, and this yields $H^\star(\phi+\psi,N)\leq\ent^\star(\phi)$; hence $\ent^\star(\phi+\psi)\leq\ent^\star(\phi)$.

\smallskip
(c) By (b), if $\ent^\star(\phi)=\ent^\star(\psi)=0$, then $\ent^\star(\phi\psi)=\ent^\star(\phi+\psi)=0$ as well.

\smallskip
(d) By Lemma \ref{ll} $\ent^\star(\phi)=0$ if and only if $\ent^\star(\phi^k)=0$ for every $k\in\N$, and analogously for $\psi$. Therefore, it follows from (c) that, if $\ent^\star(\phi)=\ent^\star(\psi)=0$, then for any polynomial $f\in \Z[X,Y]$, $\ent^\star(f(\phi,\psi))=0$.
\end{example}

\section{Adjoint algebraic entropy and algebraic entropy of the Pontryagin adjoint}\label{sec-3}

The main goal of this section is to prove in Theorem \ref{ent*=ent^} that the adjoint algebraic entropy $\ent^\star(\phi)$ of an endomorphism $\phi:G\to G$, where $G$ is an Abelian group, coincides with the algebraic entropy $\ent(\dual\phi)$ of the adjoint endomorphism $\dual\phi:\dual G\to \dual G$, where $\dual G$ is the Pontryagin dual group of $G$.

Recall that for a topological Abelian group $G$ the Pontryagin dual $\dual G$ of $G$ is the group $\mathrm{Chom}(G,\mathbb T)$ of the continuous characters of $G$ (i.e., the continuous homomorphisms $G\to \mathbb T$, where $\mathbb T=\mathbb R/\Z$) endowed with the compact-open topology \cite{P}. The Pontryagin dual of a discrete Abelian group is always compact, in which case $G^*=\mathrm{Hom}(G,\mathbb T)$.
For a subset $H$ of $G$, the \emph{annihilator} of $H$ in $\dual G$ is $H^\perp=\{\chi\in\dual G:\chi H=0\}$. If $\phi:G\to G$ is an endomorphism, its Pontryagin adjoint $\dual\phi:\dual G\to \dual G$ is defined by $\dual\phi(\chi)=\chi\circ\phi$ for every $\chi\in\dual G$.

We collect here some known facts concerning the Pontryagin duality that we will use in what follows, and which are proved in \cite{DPS}, \cite{HR} and \cite{O3}.

\begin{itemize}
\item[(i)] If $F$ is a finite Abelian group, then $\dual F\cong F$.
\item[(ii)] For a family $\{H_i:i\in I\}$ of Abelian groups, $\dual{(\bigoplus_{i\in I}H_i)}\cong\prod_{i\in I}\dual H_i$.
\item[(iii)] If $G$ is an Abelian group and $H$ a subgroup of $G$, then $H^\perp\cong\dual{(G/H)}$ and $G^*/H^\perp\cong \dual H$.
\item[(iv)] If $G$ is an Abelian group and $p$ a prime, then $(p G)^\perp=\dual G[p]$. Consequently, $\dual G[p]\cong\dual{(G/pG)}$ by (iii).
\item[(v)] If $H_1,\ldots,H_n$ are subgroups of an Abelian group $G$, then $(\sum_{i=1}^nH_i)^\perp\cong \bigcap_{i=1}^nH_i^\perp$ and $(\bigcap_{i=1}^nH_i)^\perp\cong\sum_{i=1}^n H_i^\perp$.
\end{itemize}

 We need the following lemma and proposition. The lemma is an easy to prove consequence of the known results on Pontryagin duality, applied to the adjoint endomorphism.

\begin{lemma}\label{newfacts}
Let $G$ be an Abelian group and $\phi\in\End(G)$.
\begin{itemize}
\item[(a)] If $p$ is a prime, and $\xi:\dual G[p]\to \dual{(G/pG)}$ is the isomorphism given in \emph{(iv)}, then $\xi\circ\dual\phi\restriction_{\dual G[p]}=\dual{\overline\phi}_p\circ\xi$, where $\overline\phi_p:G/pG\to G/pG$ is the endomorphism induced by $\phi$.
\item[(b)] A subgroup $H$ of $G$ is $\phi$-invariant if and only if $H^\perp$ is $\dual\phi$-invariant in $\dual G$. \hfill$\qed$
\end{itemize}
\end{lemma}

The next proposition is fundamental for the proof of Theorem \ref{ent*=ent^}.

\begin{proposition}\label{perp}
Let $G$ be an Abelian group, $H$ a subgroup of $G$ and $\phi\in\End(G)$. For every $n\in\N$, $(\phi^{-n}H)^\perp=(\dual\phi)^n H^\perp$.
\end{proposition}
\begin{proof}
We prove the result for $n=1$, that is, $(\phi^{-1}H)^\perp=\dual\phi H^\perp$. The proof for $n>1$ follows easily from this case noting that $(\dual\phi)^n=\dual{(\phi^n)}$. 

Let $\pi':G\to G/\phi^{-1}H$ and $\pi:G\to G/H$ be the canonical projections. Let $\widetilde\phi:G/\phi^{-1}H\to G/H$ be the homomorphism induced by $\phi$, and note that $\widetilde\phi$ is injective.
The proof is based on the following commuting diagram:
\begin{equation*}
\xymatrix{
G \ar@/^1.5pc/[rr]^\chi \ar[d]^\phi \ar[r]^{\pi'} & \frac{G}{\phi^{-1}H} \ar[d]^{\widetilde\phi}\ar[r]^\eta & \mathbb T \\
G \ar[r]^\pi \ar@/_3pc/[rru]_\theta & \frac{G}{H}\ar[ur]_{\xi} & 
}
\end{equation*}

Let $\chi\in(\phi^{-1}H)^\perp$. Then $\chi=\eta\circ\pi'$, where $\eta:G/\phi^{-1}H\to\mathbb T$. Since $\widetilde\phi$ is injective, and $\mathbb T$ is divisible, $\eta$ can be extended to $\xi:G/H\to \mathbb T$, i.e., $\eta=\xi\circ \widetilde\phi$. This gives $$\chi=\eta\circ \pi'=\xi\circ\widetilde\phi\circ\pi'=\xi\circ\pi\circ\phi,$$ and this shows that $\chi=\dual\phi(\theta)$, where $\theta=\xi\circ\pi\in H^\perp$. This proves the inclusion $(\phi^{-1}H)^\perp\subseteq\dual\phi H^\perp$.

Now let $\chi\in\dual\phi H^\perp$. Then $\chi=\dual\phi(\theta)=\theta\circ\phi$, where $\theta\in H^\perp$. So $\theta=\xi\circ\pi$, for some $\xi:G/H\to\mathbb T$. Take $\eta=\xi\circ\widetilde\phi$ (since $\widetilde\phi$ is injective we can think that $\eta=\xi\restriction_{G/\phi^{-1}N}$). Therefore, $$\chi=\theta\circ\phi)\xi\circ\pi\circ\phi=\xi\circ\widetilde\phi\circ\pi'=\eta\circ\pi'\in(\phi^{-1}H)^\perp.$$ This prove the inclusion $(\phi^{-1}H)^\perp\supseteq\dual\phi H^\perp$.
\end{proof}

We can now prove the main result of this section.

\begin{theorem}\label{ent*=ent^}
Let $G$ be an Abelian group and $\phi\in\End(G)$. Then $\ent^\star(\phi)=\ent(\dual\phi)$.
\end{theorem}
\begin{proof}
Let $N\in\CC(G)$. Then $F=N^\perp$ is a finite subgroup of $\dual G$ by facts (iii) and (i). By Proposition \ref{perp}, $(\phi^{-n}N)^\perp=(\dual\phi)^n F$ for every $n\in\N$. Hence, $B_n(\phi,N)^\perp=T_n(\dual \phi,F)$ for every $n\in\N_+$ by fact (v). It follows that $|C_n(\phi,N)|=|\dual{C_n(\phi,N)}|=|B_n(\phi,N)^\perp|=|T_n(\dual\phi,F)|$ for every $n\in\N_+$, and this concludes the proof.
\end{proof}

Let us denote by $\overline H$ the closure of a subgroup $H$ of a topological group $K$. So, if $G$ is an Abelian group and $H\leq \dual G$, $\overline H$ denotes the closure of $H$ in the compact-open topology.

Since $t(\dual G)$ and $\overline{t(\dual G)}$ are $\dual\phi$-invariant subgroups of $\dual G$, we have the following result.

\begin{corollary}\label{entphi*}
Let $G$ be an Abelian group and $\phi\in\End(G)$. Then:
\begin{itemize}
\item[(a)]$\ent(\dual\phi)=\ent(\dual\phi\restriction_{t(\dual G)})=\ent(\dual\phi\restriction_{\overline{t(\dual G)}})$;
\item[(b)]$\ent^\star(\phi)=0$ if and only if $\ent(\dual\phi\restriction_{\dual G[p]})=0$ for every prime $p$.
\end{itemize}
\end{corollary}
\begin{proof}
(a) Since $t(\dual G)\leq\overline{t(\dual G)}\leq G$, and $\ent(\phi)=\ent(\phi\restriction_{t(\dual G)})$, we have the equality in (a), since $\ent(-)$ is monotone under taking invariant subgroups (see \cite[fact (g) in Section 1]{DGSZ}).

\smallskip
(b) Follows from (a) and Lemma \ref{1.18}.
\end{proof}

As a consequence of Corollary \ref{entphi*} we obtain another proof of Lemma \ref{narrow}. Indeed, if $G$ is a narrow Abelian group, then $G/pG$ is finite for every prime $p$. By facts (i), (iii) and (iv) at the beginning of Section \ref{sec-3}, $\dual G[p]\cong \dual{(G/pG)}$ is finite; if $\phi\in\End(G)$, then $\ent(\dual\phi\restriction_{\dual G[p]})=0$ for every prime $p$. By Corollary \ref{entphi*} $\ent^\star(\phi)=0$.

\medskip
As another corollary of Theorem \ref{ent*=ent^} we get a second proof of Example \ref{tf,fr->ent0}(b), using the Pontryagin duality.
Let $G$ be a torsion-free Abelian group of finite rank. We show that $\ent^\star(G)=0$.
So let $\phi\in\End(G)$. By Theorem \ref{ent*=ent^} $\ent^\star(\phi)=\ent(\dual\phi)$ and so we have to prove that $\ent(\dual\phi)=0$.
By Lemma \ref{1.18} it suffices to show that $\ent(\dual\phi\restriction_{\dual G[p]})=0$ for every prime $p$. To this end we verify that $\dual G[p]$ is finite for every prime $p$.
By hypothesis there exists $n\in\N$ such that $r_0(G)\leq n$. By facts (iv) and (iii) $\dual G[p]= (p G)^\perp\cong\dual{(G/pG)}$. Since $r_0(G)\leq n$ and $G/pG$ is $p$-bounded, it follows that $r_p(G/pG)\leq n$, where $r_p(-)$ denotes the dimension over the field $\mathbb F_p$.
Since $G/pG$ is finite, also $\dual G[p]$ is finite. Therefore $\ent(\dual\phi\restriction_{\dual G[p]})=0$.

\medskip
Note that for an Abelian group $G$, setting $G^1$ and $\phi^1$ as in \eqref{phi^1}, 
$$(G^1)^\perp=\left(\bigcap\{F\leq G:F\ \text{finite}\}\right)^\perp=\overline{\sum\{F^\perp\leq\dual G:F^\perp\ \text{finite}\}}=\overline{t(\dual G)}.$$
Consequently, for an Abelian group $G$ one has $G^1=0$ exactly when $t(\dual G)$ is dense in $\dual G$.
So for the Abelian group $G$ we have $\ent^\star(\phi^1)=\ent^\star(\phi)=\ent(\dual\phi)=\ent(\dual\phi\restriction_{\overline{t(\dual G)}})$.

\section{Applications to the Bernoulli shifts and the Addition Theorem}\label{sec6}

This section is devoted to two relevant applications of Theorem \ref{ent*=ent^}. The first application is the computation of the adjoint algebraic entropy of particularly important endomorphisms of Abelian groups, namely, the Bernoulli shifts introduced in Example \ref{bernoulli-example}. The second application is the proof of the Addition Theorem for the adjoint algebraic entropy; here the unavoidable restriction to bounded Abelian groups arises, in analogy with the restriction to torsion Abelian groups needed for the Addition Theorem for the algebraic entropy.

\medskip
The next proposition shows that in the Pontryagin duality the adjoint of a Bernoulli shift on the direct sum is a Bernoulli shift on the direct product, and this functor reverses the direction of the shift. In the notation of Example \ref{bernoulli-example}, we have the following

\begin{proposition}\label{dualbeta}
For $K=\Z(p)$, where $p$ is a prime,
\begin{itemize}
\item[(a)]$\dual{(\beta_K^\oplus)}={}_K\beta$,
\item[(b)]$\dual{( {}_K\beta^\oplus)}=\beta_K$, and
\item[(c)]$\dual{(\overline\beta_K^\oplus)}=(\overline\beta_K)^{-1}$.
\end{itemize}
\end{proposition}
\begin{proof}
Let $\chi=(a_0,a_1,\ldots)\in K^\N=\dual{(K^{(N)})}$, and consider the $i$-th canonical vector $e_i=(\underbrace{0,\ldots,0}_i,1, 0,\ldots,0,\ldots)$ of the base of $K^{(\N)}$ for $i\in\N$. Then $\chi(e_i)=a_i$ for every $i\in\N$, and hence $\chi(x)=\sum_{i\in\N}a_i x_i$ for every $x=(x_i)_{i\in\N}\in K^{(\N)}$.

\smallskip
(a) Consider $\dual{(\beta_K^\oplus)}:K^\N\to K^\N$. Then $\dual{(\beta_K^\oplus)}(\chi)=\chi\circ\ \beta_K^\oplus=(a_1,a_2,\ldots)= {}_K\! \beta(\chi)$,  because $\chi\circ\ \beta_K(e_0)=0$ and $\chi\circ\ \beta_K(e_i)=a_{i+1}$ for every $i\in\N_+$.

\smallskip
(b) Consider $\dual{( {}_K\beta^\oplus)}:K^\N\to K^\N$. Then $\dual{({}_K\beta^\oplus)}(\chi)=\chi\circ {}_K\beta^\oplus=(0,a_0,a_1,\ldots)=\beta_K(\chi)$, because $\chi\circ {}_K\beta(e_0)=0$ and $\chi\circ {}_K\beta(e_i)=a_{i-1}$ for every $i\in\N_+$.

\smallskip
(c) Let $\chi=(\ldots,a_{-1},\underbrace{a_0}_0,a_1,\ldots)\in K^\Z=\dual{(K^{(\Z)})}$, and for $i\in\Z$ consider the $i$-th canonical vector $e_i=(\ldots,0,\ldots,0,\underbrace{1}_i, 0,\ldots,0,\ldots)$ of the base of $K^{(\Z)}$. Then $\chi(e_i)=a_i$ for every $i\in\Z$, and hence $\chi(x)=\sum_{i\in\Z}a_i x_i$ for every $x=(x_i)_{i\in\Z}\in K^{(\Z)}$.
Consider $\dual{(\overline\beta_K^\oplus)}:K^\Z\to K^\Z$. Then $\dual{(\overline\beta_K^\oplus)}(\chi)=\chi\circ\ \overline\beta_K^\oplus=(\ldots,a_0,\underbrace{a_1}_0,a_2,\ldots)= (\overline\beta_K)^{-1}(\chi)$,  because $\chi\circ\ \overline\beta_K(e_i)=a_{i+1}$ for every $i\in\Z$.
\end{proof}

In contrast with what happens for the algebraic entropy, that is, $\ent( {}_K\beta^\oplus)=0$ and $\ent(\beta_K^\oplus)=\ent(\overline\beta_K^\oplus)=\log|K|$ (see \cite[Example 1.9]{DGSZ}), we have:

\begin{proposition}\label{bernoulli}
For $K=\Z(p)$, where $p$ is a prime, $\ent^\star(\beta_K^\oplus)=\ent^\star( {}_K\beta^\oplus)=\ent^\star(\overline\beta_K^\oplus)=\infty.$
\end{proposition}
\begin{proof}
It is known from \cite[Corollary 6.5]{AGB} that $\ent(\beta_K)=\ent( {}_K\beta)=\ent(\overline\beta_K)=\infty$. Consequently, by Theorem \ref{ent*=ent^} and Proposition \ref{dualbeta} $\ent^\star(\beta_K^\oplus)=\ent({}_K\beta)=\infty$, $\ent^\star({}_K\beta^\oplus)=\ent(\beta_K)=\infty$ and $\ent^\star(\overline\beta_K^\oplus)=\ent((\overline\beta_K)^{-1})=\infty$. 
\end{proof}

As an application of Proposition \ref{bernoulli} we can now give the example which witnesses that the continuity of the adjoint algebraic entropy for the inverse limits does not hold true. As noted before, this comment is inspired by the continuity of the algebraic entropy for direct limits of invariant subgroups (see property (D) above).

\begin{example}\label{no-lim}
Let $p$ be a prime, $G=\Z(p)^\N$ and consider $\beta_{\Z(p)}:G\to G$. For every $i\in\N$, let $$H_i=\underbrace{0\times\ldots\times 0}_i\times\Z(p)^{\N\setminus\{0,\ldots,i-1\}}\leq G.$$ Each $H_i$ is $\beta_{\Z(p)}$-invariant. The induced endomorphism $\overline{\beta_{\Z(p)}}_i:G/H_i\to G/H_i$ has $\ent^\star(\overline{\beta_{\Z(p)}}_i)=0$, since $G/H_i$ is finite. Moreover, $G=\displaystyle\lim_{\longleftarrow}G/H_i$, as $\{(G/H_i,p_i)\}_{i\in\N}$ is an inverse system, where $p_i:G/H_{i+1}\cong\Z(p)^{i+1}\to G/H_{i}\cong \Z(p)^i$ is the canonical projection for every $i\in\N$.
By Proposition \ref{bernoulli} $\ent^\star(\beta_{\Z(p)})=\infty$, while $\sup_{i\in\N}\ent^\star(\overline{\beta_{\Z(p)}}_i)=0$.
\end{example}

The dual behaviour of $\ent^\star(-)$ with respect to $\ent(-)$ fails also for other aspects. Indeed, for example \cite[Lemma 1.4]{DGSZ} states in particular that for an Abelian group $G$ and $\phi\in\End(G)$, if $G=T(\phi,F)$ for some finite subgroup $F$ of $G$, then $\ent(\phi)=H(\phi,F)$. The dual version of this lemma fails, as shown by the next example.

\begin{example}\label{not-dual-ex}
Consider $K=\Z(p)$ and ${}_K\beta^\oplus:G\to G$, where $G= K^{(\N)}$. For $N=\{0\}\oplus K^{(\N_+)}$, we have $B( {}_K\beta^\oplus,N)=\{0\}$ and so $C( {}_K\beta^\oplus,N)=G$. But $\ent^\star( {}_K\beta^\oplus)=\infty$ by Proposition \ref{bernoulli}, while $H^\star({}_K\beta^\oplus,N)=\log|K|$. In particular, this shows that $\ent^\star( {}_K\beta^\oplus)\neq H^\star( {}_K\beta^\oplus,N)$.
\end{example}

The reason of this behaviour relies in the fact that $G$ has too many subgroups of finite index, as shown by Lemma \ref{>c}. For a further failure of the dual behaviour of $\ent^\star(-)$ with respect to $\ent(-)$ see the next Remark \ref{6.7}.

\bigskip
We pass now to consider the Addition Theorem for the adjoint algebraic entropy of an endomorphisms $\phi$ of a torsion Abelian group $G$ (see \cite[Theorem 3.1]{DGSZ}, or fact (c) at the beginning of Section \ref{basic}).

The following result is the counterpart of the Addition Theorem for the adjoint algebraic entropy of endomorphisms of bounded Abelian groups. To prove it we apply Theorem \ref{ent*=ent^} and then the Addition Theorem for the algebraic entropy of endomorphisms of torsion Abelian groups (which is property (C) above).

\begin{proposition}\label{AT*}
Let $G$ be a bounded Abelian group, $\phi\in\End(G)$ and $H$ a $\phi$-invariant subgroup of $G$. Then $\ent^\star(\phi)=\ent^\star(\phi\restriction_H)+\ent^\star(\overline\phi)$, where $\overline\phi:G/H\to G/H$ is the endomorphism induced by $\phi$.
\end{proposition}
\begin{proof}
For the group $G$ we have the following diagram:
\begin{equation*}
\xymatrix{
0\ar[r] & H\ar[r] \ar[d]^{\phi\restriction_{H}} & G\ar[r]\ar[d]^\phi & G/H\ar[r]\ar[d]^{\overline\phi} & 0\\
0 \ar[r] & H\ar[r] & G\ar[r] & G/H\ar[r] & 0
}
\end{equation*}
By the Pontryagin duality we have:
\begin{equation*}
\xymatrix{
0 & \dual H \ar[l] & \dual G \ar[l] & \dual{(G/H)} \ar[l] &\ar[l] 0\\
0 & \dual H \ar[l]\ar[u]_{\dual{(\phi\restriction_H)}} & \dual G \ar[l] \ar[u]_{\dual\phi} & \dual{(G/H)} \ar[l] \ar[u]_{\dual{\overline\phi}} &\ar[l] 0
}
\end{equation*}
Since $\dual{(G/H)}\cong H^\perp$, $\dual H \cong G/H^\perp$ and $H^\perp$ is $\dual\phi$-invariant respectively by fact (iii) at the beginning of Section \ref{sec-3} and Lemma \ref{newfacts}(b), we have the following diagram, where $\dual\phi\restriction_{H^\perp}$ is conjugated by an isomorphism to $\dual{\overline\phi}$, and $\overline{\dual\phi}$ is conjugated by an isomorphism to $\dual{(\phi\restriction_H)}$.
\begin{equation*}
\xymatrix{
0 & \dual G/H^\perp \ar[l]  & \dual G \ar[l] & H^\perp \ar[l] &\ar[l] 0\\
0 & \dual G/H^\perp \ar[l] \ar[u]_{\overline{\dual\phi}}& \dual G \ar[l] \ar[u]_{\dual\phi} & H^\perp \ar[l] \ar[u]_{\dual\phi\restriction_{H^\perp}} &\ar[l] 0
}
\end{equation*}
Since $G$ is bounded, $\dual G$ is bounded as well, and in particular it is a torsion Abelian group. So it is possible to apply the Addition Theorem for the algebraic entropy \cite[Theorem 3.1]{DGSZ}, which gives $\ent(\dual\phi)=\ent(\overline{\dual\phi})+\ent(\phi\restriction_{H^\perp})$. Since the algebraic entropy is preserved under taking conjugation by isomorphisms, $\ent(\dual\phi)=\ent(\dual{(\phi\restriction_H)})+\ent(\dual{\overline\phi})$. By Theorem \ref{ent*=ent^} $\ent^\star(\phi)=\ent^\star(\phi\restriction_H)+\ent^\star(\overline\phi)$.
\end{proof}

 Lemma \ref{poorAT} shows that the Addition Theorem for the adjoint algebraic entropy holds for every endomorphism $\phi$ of any Abelian group, when $\phi$ is direct product of finitely many endomorphisms. On the other hand, Example \ref{noAT*} shows that the monotonicity of the adjoint algebraic entropy under taking invariant subgroups fails even for torsion Abelian groups, so the Addition Theorem for the adjoint algebraic entropy does not hold in general for every endomorphism of Abelian groups.
Moreover, it is not possible to weaken the hypothesis of Proposition \ref{AT*}, that is, the Addition Theorem for the adjoint algebraic entropy cannot be proved with the argument used in the proof of Proposition \ref{AT*} out of the class of bounded Abelian groups. Indeed, for an Abelian group $G$, the Pontryagin dual $\dual G$ is compact, and if $\dual G$ is torsion, then it is bounded (see \cite{HR}), so $G$ is bounded as well. Since in Proposition \ref{AT*} we apply the Addition Theorem for the algebraic entropy, which holds for endomorphisms of torsion Abelian groups, $\dual G$ has to be torsion, hence bounded, and so, if we want to use this argument, $G$ has to be bounded.

\begin{corollary}\label{phi^nG}
Let $G$ be a bounded Abelian group and $\phi\in\End(G)$. For every $n\in\N$, $\ent^\star(\phi)=\ent^\star(\phi\restriction_{\phi^n G})$.
\end{corollary}
\begin{proof}
For $n\in\N_+$, consider the induced endomorphism $\overline\phi_n:G/\phi^n G\to G/\phi^n G$. Since $(\overline\phi_n)^n=0_{G/\phi^n G}$, as observed in Example \ref{multiplication}(a) $\ent^\star(\overline\phi_n)=0$, and hence by Proposition \ref{AT*} $\ent^\star(\phi)=\ent^\star(\phi\restriction_{\phi^n G})$.
\end{proof}

\begin{remark}\label{6.7}
For the algebraic entropy there is a natural reduction to injective endomorphisms. Indeed, for a torsion Abelian group $G$ and $\phi\in\End(G)$, in \cite{DGSZ} the $\phi$-torsion subgroup of $G$ was introduced, that is, $t_\phi(G)=\{x\in G:T(\phi,x)\ \text{finite}\}$; moreover, it was proved that $\ent(\phi)=\ent(\overline\phi)$, where $\overline\phi:G/t_\phi(G)\to G/t_\phi(G)$ is the induced endomorphism. Since $t_\phi(G)\geq \ker_\infty\phi:=\bigcup_{n\in\N_+}\ker\phi^n$ \cite[Lemma 2.3]{DGSZ}, and $\ker_\infty\phi$ is the minimum subgroup of $G$ such that the endomorphism induced by $\phi$ on $G/\ker_\infty\phi\to G/\ker_\infty\phi$ is injective, it follows that $\overline \phi$ is injective.

So, if the adjoint algebraic entropy was dual with respect to the algebraic entropy, then it would have an analogous reduction to surjective endomorphism. This seems to be suggested also by Corollary \ref{phi^nG}, since the dual subgroup with respect to $\ker_\infty\phi$ is $\Im_\infty\phi:=\bigcap_{n\in\N_+}\phi^nG$. 
But $\ent^\star(\phi)$ does not coincide with $\ent^\star(\phi\restriction_{\Im_\infty\phi})$ in general, even for $p$-bounded Abelian groups; indeed, for $G=\Z(p)^{(\N)}$ and the right Bernoulli shift $\beta_{\Z(p)}^\oplus$, we have $\ent^\star(\beta_{\Z(p)}^\oplus)=\infty$ (see Proposition \ref{bernoulli}) while $\Im_\infty\phi=0$.
\end{remark}

\section{Dichotomy}\label{dichotomy}

The value of the adjoint algebraic entropy of the Bernoulli shifts is infinite, as proved in Proposition \ref{bernoulli}, and this suggests that the behaviour of $\ent^\star(\phi)$ presents a dichotomy, namely, for every Abelian group $G$ and every $\phi\in\End(G)$, either $\ent^\star(\phi)=0$ or $\ent^\star(\phi)=\infty$. 

\medskip
The following lemma permits to reduce to the case of $p$-bounded Abelian groups, with $p$ a prime.

\begin{lemma}\label{red-to-p}
Let $G$ be an Abelian group and $\phi\in\End(G)$. If $\ent^\star(\phi)>0$, then there exists a prime $p$ such that $\ent^\star(\overline\phi_p)>0$, where $\overline \phi_p:G/p G\to G/p G$ is induced by $\phi$. In particular, if $0<\ent^\star(\phi)<\infty$, then there exists a prime $p$ such that $0<\ent^\star(\overline\phi_p)<\infty$.
\end{lemma}
\begin{proof}
By Corollary \ref{entphi*}(b), $\ent^\star(\phi)>0$ is equivalent to the existence of a prime $p$ such that $\ent(\dual\phi\restriction_{\dual G[p]})>0$. By facts (iii) and (iv) at the beginning of Section \ref{sec-3}, $\dual G[p]=(pG)^\perp\cong\dual{(G/pG)}$, and by Lemma \ref{newfacts}(a), $\dual{\overline\phi_p}$ is conjugated to $\dual\phi\restriction_{\dual G[p]}$ by an isomorphism; therefore, by Theorem \ref{ent*=ent^}, $\ent^\star(\overline\phi_p)=\ent(\dual{\overline\phi_p})=\ent(\dual\phi\restriction_{\dual G[p]})>0$, where the second equality follows from property (A) at the beginning of Section \ref{basic}.
If $\ent^\star(\phi)<\infty$, then $\ent^\star(\overline\phi_p)<\infty$ for every prime $p$, as $\ent^\star(\overline\phi_p)\leq\ent^\star(\phi)$ by Lemma \ref{quotient}.
\end{proof}

The preceding lemma leads us to the investigation of linear transformations of vector spaces over the Galois field with $p$ elements $\mathbb F_p$ (where $p$ is a prime).
So let $V$ be a vector space over the field $\mathbb F_p$, and let $\phi:V\to V$ be a linear transformation. Let $\mathbb F_p[X]$ be the ring of polynomials in the variable $X$ with coefficients in $\mathbb F_p$. Following Kaplansky's book \cite[Chapter 12]{K}, for $v\in V$ and $f\in \mathbb F_p[X]$, define $$f(X)\cdot v= f(\phi)(v).$$
In detail, if $f(X)=a_0+a_1 X+\ldots + a_m X^m$, then $$f(X)\cdot v=a_0 v+a_1 \phi(v)+\ldots+ a_m\phi^m(v).$$
To underline the role of $\phi$ in this definition, we consider $\mathbb F_p[\phi]$ instead of $\mathbb F_p[X]$ (they are isomorphic in case $\phi$ is not algebraic).
This definition makes $V$ a $\mathbb F_p[\phi]$-module. When we consider $V$ as a $\mathbb F_p[\phi]$-module, we write $V_\phi$.

We say that an element $v\in V_\phi$ is \emph{torsion}, if there exists a non-zero $f \in \mathbb F_p[X]$ such that $f(\phi)(v)=0$. Furthermore, $V_\phi$ is \emph{torsion} if each element of $V_\phi$ is torsion; this means that $\phi$ is locally algebraic, since every $v\in V$ is the root of a polynomial in $\mathbb F_p[X]$. Moreover, $V_\phi$ is \emph{bounded} if there exists $0\neq f \in \mathbb F_p[X]$ such that $f(\phi)V=0$, i.e., $f(\phi)=0$. This amounts to say that $\phi$ is algebraic over $\mathbb F_p$.

\begin{lemma}\label{ent*:pol:phi}
\begin{itemize}
\item[(a)]Let $G$ be an Abelian group and $\phi\in\End(G)$. If $f\in\Z[X]$, then $\ent^\star(f(\phi))\leq\deg f\cdot\ent^\star(\phi)$.
\item[(b)]Let $V$ be a vector space over the field $\mathbb F_p$, and $\phi:V\to V$ a linear transformation. If $f\in \mathbb F_p[X]$, then $\ent^\star(f(\phi))\leq\deg f\cdot\ent^\star(\phi)$.
\end{itemize}
\end{lemma}
\begin{proof}
(a) Let $f=a_0+a_1 X+\ldots+ a_k X^k$, where $k=\deg f$ and $a_0,a_1,\ldots,a_k\in \Z$. Let $n\in\N_+$ and $N\in\CC(G)$. Then an easy check shows that $$B_n(f(\phi),N)\geq B_{kn}(\phi,N).$$
Consequently, also by \eqref{Cnkf=Cnfk} in Lemma \ref{ll}, $$\card{C_n(f(\phi),N)}\leq\card{C_{kn}(\phi,N)}=\card{C_n(\phi^k,B_k(\phi,N))}.$$ Hence, $H^\star(f(\phi),N)\leq H^\star(\phi^k,B_k(\phi,N))\leq\ent^\star(\phi^k)$, and so $\ent^\star(f(\phi))\leq\ent^\star(\phi^k)$. By Lemma \ref{ll}, $\ent^\star(\phi^k)=k\cdot\ent^\star(\phi)$.

\smallskip
(b) Is obtained verbatim from (a), replacing $a_i\in\Z$ with $a_i\in \mathbb F_p$.
\end{proof}

In general, for an endomorphism $\phi$ of an Abelian group $G$, $\ent(\phi)>0$ need not imply $\ent^\star(\phi)>0$; but for bounded Abelian groups this holds in the following strong form.

\begin{lemma}\label{ent*>ent}
Let $G$ be a bounded Abelian group and $\phi\in\End(G)$. If $\ent(\phi)>0$, then $\ent^\star(\phi)=\infty$.
\end{lemma}
\begin{proof}
By Lemma \ref{1.18} there exists a prime $p$ such that $\ent(\phi_p)>0$, where $\phi_p=\phi\restriction_{G[p]}$. By Fact \ref{non-torsion->ent*=infty} this is equivalent to the existence of an infinite trajectory $T(\phi_p,x)=\sum_{n\in\N}\phi_p^n\hull x$ for some $x\in G[p]$. Then $\phi\restriction_{T(\phi_p,x)}$ is a right Bernoulli shift, and so $\ent^\star(\phi_p\restriction_{T(\phi_p,x)})=\infty$ by Proposition \ref{bernoulli}. Now Proposition \ref{AT*} yields $\ent^\star(\phi_p)=\infty$ as well, and hence $\ent^\star(\phi)=\infty$.
\end{proof}

In particular, Lemma \ref{ent*>ent} shows that $\ent^\star(\phi)\geq\ent(\phi)$ for any endomorphism $\phi$ of a bounded Abelian group $G$. Boundedness is essential in this lemma, since the right Bernoulli shift $\beta_{\Z(p^\infty)}$ of the group $\Z(p^\infty)^{(\N)}$ has $\ent(\beta_{\Z(p^\infty)})=\infty$, while $\ent^\star(\beta_{\Z(p^\infty)})=0$ by Example \ref{example1} since $\Z(p^\infty)$ is divisible. 

\begin{lemma}\label{alg->qp}
Let $p$ be a prime, $V$ a vector space over $\mathbb F_p$ and $\phi:V\to V$ a linear transformation. If $\phi$ is algebraic, then $\phi$ is quasi-periodic.
\end{lemma}
\begin{proof}
Let $f\in\mathbb F_p[X]$ be the minimal polynomial such that $f(\phi)=0$. In particular, $f$ is irreducible. Then $\frac{\mathbb F_p[X]}{(f)}$ is a finite field of cardinality $p^{\deg f}$. Hence, $\phi^{p^{\deg f}}=\phi$.
\end{proof}

The proof of the next theorem contains the core of the dichotomy for the adjoint algebraic entropy.

\begin{theorem}\label{dich}
Let $p$ be a prime, $V$ a vector space over $\mathbb F_p$ and $\phi:V\to V$ a linear transformation. The following conditions are equivalent:
\begin{itemize}
\item[(a)] $\phi$ is algebraic (i.e., $V_\phi$ is a bounded $\mathbb F_p[\phi]$-module);
\item[(b)] $\phi$ is quasi-periodic;
\item[(c)] $\ent^\star(\phi)=0$;
\item[(d)] $\ent^\star(\phi)<\infty$.
\end{itemize}
\end{theorem}
\begin{proof}
(a)$\Rightarrow$(b) is Lemma \ref{alg->qp}, 
(b)$\Rightarrow$(c) is Example \ref{qp->ent*=0} and
(c)$\Rightarrow$(d) is obvious.

\smallskip
(d)$\Rightarrow$(a) Assume by way of contradiction $\phi$ is not algebraic, that is, that $V_\phi$ is not bounded. We prove that $\ent^\star(\phi)=\infty$.
 
If $V_\phi$ is not torsion (i.e., $\phi$ is not locally algebraic), by Proposition \ref{non-torsion->ent*=infty} $\ent(\phi)>0$ and so Lemma \ref{ent*>ent} gives $\ent^\star(\phi)=\infty$.
Thus let us assume that $V_\phi$ is torsion. 

First suppose that the module $V_\phi$ is not reduced. Then there exist $f\in \mathbb F_p[X]$ irreducible and a $\mathbb F_p$-independent family $\{v_n\}_{n\in\N}\subseteq V$ such that $$f(\phi)(v_0)=0, f(\phi)(v_1)=v_0,\ldots,f(\phi)(v_{n+1})=v_n,\ldots\ .$$ Then $f(\phi)$ is a left Bernoulli shift on $\hull{v_n:n\in\N}$, and $\ent^\star(f(\phi))=\infty$ by Proposition \ref{bernoulli}. By Lemma \ref{ent*:pol:phi}(b) $\ent^\star(\phi)=\infty$ as well.

Finally, suppose that $V_\phi$ is reduced torsion unbounded. It is well-known that there exist infinitely many irreducible polynomials in $\mathbb F_p[X]$, and so there exist irreducible polynomials in $\mathbb F_p[X]$ of arbitrarily large degree. 
Then $V_\phi$ contains $\bigoplus_{n\in\N}V_n$, where $V_n=\K[X]/(f_n)$, with $f_n\in\K[X]$ and $d_n=\deg(f_n)$ for every $n\in\N$; moreover, $d_0<d_1<\ldots<d_n<\ldots$, and either the $f_n$'s are pairwise distinct and irreducible, or $f_n=f^{r_n}$ for some $f\in\K[X]$ irreducible with $r_n\in\N$ for every $n\in\N$. Since each $V_n$ is a cyclic $\K[\phi]$-module, $V_n=T(\phi,v_n)$ for some $v_n\in V_n$, and so in particular each $V_n$ is a $\phi$-invariant subspace of $V$ (with $\dim(V_n)=d_n$) and an indecomposable $\K[X]$-module.

Assume without loss of generality that $V_\phi=\bigoplus_{n\in\N}V_n$. Let $\phi_n=\phi\restriction_{V_n}$ for every $n\in\N$.
Consider $\dual\phi:\dual V\to \dual V$. By facts (i) and (ii) in Section \ref{sec-3}, $\dual V\cong \prod_{n\in\N}\dual V_n$ and $\dual V_n\cong V_n$ for every $n\in\N$. Moreover, $\dual V_m\cong(\bigoplus_{n\neq m}V_n)^\perp$ and so $\dual V_m$ is $\dual\phi$-invariant by Lemma \ref{newfacts}(b), and $\dual\phi\restriction_{\dual V_m}=\dual\phi_m$.
Since $(V_n)_\phi$ is indecomposable, $(\dual V_n)_{\dual\phi}$ is indecomposable as well. 
Furthermore, being finite (and so bounded), $(\dual V_n)_{\dual\phi}$ is direct sum of cyclic $\K[\dual\phi]$-modules \cite{K}, and so $(\dual V_n)_{\dual\phi}$ has to be itself a cyclic $\K[\dual\phi]$-module. 
In other words, $(\dual V_n)_{\dual\phi}= \K[\dual\phi]\cdot \nu_n$, that is $\dual V_n=T(\dual\phi,\nu_n)$, for some $\nu_n\in\dual V_n$, for every $n\in\N$.
Let $\nu=(\nu_n)_{n\in\N}\in \prod_{n\in\N}\dual V_n\cong\dual V$. Then $\nu$ is not quasi-periodic for $\dual\phi$. Indeed, for $s<t$ in $\N$, take $n\in\N$ with $s<t<d_n$; then $(\dual\phi_n)^s(\nu_n)\neq(\dual\phi_n)^t(\nu_n)$ and so $(\dual\phi)^s(\nu)\neq(\dual\phi)^t(\nu)$. This shows that $\dual\phi$ is not locally quasi-periodic, and so $\ent(\dual\phi)>0$ by Proposition \ref{non-torsion->ent*=infty}. This implies $\ent(\dual\phi)\geq\log p$.

It follows that $\ent(\dual\phi)=\infty$. In fact, there exists a partition $\N=\dot\bigcup_{i\in\N}N_i$ of $\N$, where each $N_i$ is infinite. Then $\dual V\cong\prod_{n\in\N}\dual V_n\cong\prod_{i\in\N}W_i$, where $W_i=\prod_{n\in N_i}\dual V_n$ is $\dual\phi$-invariant and has the same properties of $\dual V$ for every $i\in\N$. By the previous part of the proof $\ent(\dual\phi\restriction_{W_i})\geq\log p$ for every $i\in\N$ and so $\ent(\dual\phi)\geq\sum_{i\in\N}\ent(\dual\phi\restriction_{W_i})=\infty$.
By Theorem \ref{ent*=ent^}, $\ent^\star(\phi)=\infty$ as well.
\end{proof}

While the value of the algebraic entropy of a linear transformation $\phi:V\to V$ of the $\K$-vector space $V$ distinguishes between the torsion and the non-torsion structure of $V_\phi$, when $V_\phi$ is torsion the value of the adjoint algebraic entropy distinguishes between the bounded and the unbounded case. In particular, $V_\phi$ is unbounded if and only if $\ent^\star(\phi)=\infty$.

\medskip
Applying Lemma \ref{red-to-p} and Theorem \ref{dich} we get the required dichotomy for the adjoint algebraic entropy of endomorphisms of Abelian groups:

\begin{theorem}\label{dich-2}
If $G$ is an Abelian group and $\phi\in\End(G)$, then either $\ent^\star(\phi)=0$ or $\ent^\star(\phi)=\infty$.
\end{theorem}
\begin{proof}
Assume that $\ent^\star(\phi)>0$. By Lemma \ref{red-to-p} there exists a prime $p$ such that the induced endomorphism $\overline\phi_p: G/pG\to G/pG$ has $\ent^\star(\overline\phi_p)>0$. Theorem \ref{dich} yields $\ent^\star(\overline\phi_p)=\infty$. Then $\ent^\star(\phi)=\infty$ by Lemma \ref{quotient}.
\end{proof}

Now we collect together the properties characterizing the endomorphisms with zero adjoint algebraic entropy. It is worthwhile to compare these properties with those of the algebraic entropy in Proposition \ref{non-torsion->ent*=infty}.

\begin{corollary}\label{ent*=0}
Let $G$ be an Abelian group and $\phi\in\End(G)$.
Denote by $\overline\phi_p:G/pG\to G/pG$ the endomorphism induced by $\phi$ for every prime $p$; then the following conditions are equivalent:
\begin{itemize}
\item[(a)] $\ent^\star(\phi)=0$;
\item[(b)] for every prime $p$ there exists a (monic) polynomial $f_p\in \Z[X]$ such that $f_p(\phi)(G)\subseteq p G$; 
\item[(c)] $\ent^\star(\overline\phi_p)=0$ for every prime $p$;
\item[(d)] $\overline\phi_p$ is algebraic (i.e., $(G/p G)_{\overline\phi_p}$ is a bounded $\K[\overline\phi_p]$-module) for every prime $p$;
\item[(e)] $\overline\phi_p$ is quasi-periodic for every prime $p$;
\item[(f)] $\ent^\star(\overline\phi_p)<\infty$ for every prime $p$;
\item[(g)] $\ent^\star(\phi)<\infty$.
\end{itemize}
\end{corollary}
\begin{proof}
The equivalence (c)$\Leftrightarrow$(d)$\Leftrightarrow$(e)$\Leftrightarrow$(f) is in Theorem \ref{dich}, 
(a)$\Rightarrow$(c) follows from Lemma \ref{quotient} and (c)$\Rightarrow$(a) from Lemma \ref{red-to-p}, while
(g)$\Rightarrow$(f) follows from Lemma \ref{quotient} and (f)$\Rightarrow$(g) from Lemma \ref{red-to-p} and Theorem \ref{dich}; finally (b)$\Leftrightarrow$(d) is clear from the definition and a simple computation.
\end{proof}

Consequently, we have $\ent^\star(\phi)=\infty$ if and only if there exists a prime $p$ such that $\ent^\star(\overline\phi_p)=\infty$.
By the remark after Proposition \ref{C=C_n->ent*=0}, $\ent^\star(\phi)=\infty$ if and only if there exists a prime $p$ such that $G/pG$ has a countable family $\{\overline N_k\}_{k\in\N}\subseteq \CC(G/pG)$, such that $H^\star(\overline\phi_p,\overline N_k)$ converges to $\infty$.

\medskip
Finally, as a consequence of Theorem \ref{ent*=ent^} and of the dichotomy on the value of the adjoint algebraic entropy of group endomorphisms of Abelian groups, we can deduce an analogous result on the value of the algebraic entropy of continuous endomorphisms of compact Abelian groups:

\begin{corollary}\label{dich-compcont}
Let $K$ be a compact Abelian group and $\psi:K\to K$ a continuous endomorphism. Then either $\ent(\psi)=0$ or $\ent(\psi)=\infty$.
\end{corollary}
\begin{proof}
Let $\phi=\dual\psi$. By Pontryagin duality $\dual\phi=\psi$. So Theorem \ref{ent*=ent^} yields $\ent(\psi)=\ent^\star(\phi)$ and Theorem \ref{dich-2} concludes the proof.
\end{proof}

This corollary gives as a consequence the following surprising algebraic criterion for continuity of the endomorphisms of compact Abelian groups: 

\begin{corollary}\label{dich-compcont*}
Let $K$ be a compact Abelian group and $\psi:K\to K$ an endomorphism with $0< \ent(\psi)<\infty$. Then $\psi $ is discontinuous.  \hfill$\qed$
\end{corollary}

In particular, Corollary \ref{dich-compcont} answers in a strong negative way Problem 7.1 in \cite{AGB}, which asked if there exists a continuous endomorphism of $K^\N$, where $K$ is a finite Abelian group, with positive finite algebraic entropy. In that paper it was proved that a special class of endomorphisms (namely, the generalized shifts) of a direct product of a fixed finite Abelian group has algebraic entropy either zero or infinite. In particular, these are continuous endomorphisms of compact Abelian groups. So Theorem 1.3 in \cite{AGB} is covered by Corollary \ref{dich-compcont}. Nevertheless, one cannot obtain that theorem as a corollary of Theorem \ref{dich}, since we used in the proof of Theorem \ref{dich} the fact that the Bernoulli shifts considered on direct products have infinite algebraic entropy.

For further connection between the adjoint algebraic entropy and the topological entropy of endomorphism of compact Abelian groups see the forthcoming paper \cite{DG}.

\end{document}